\documentclass[draft, 12pt]{article}

\usepackage{amsthm}
\usepackage{amsmath}
\usepackage{amssymb}
\usepackage{amscd}                 
\usepackage[all]{xy}               
\usepackage{stmaryrd}              

\theoremstyle{plain}

\newtheorem{thm}{Theorem}[section]
\newtheorem{cor}[thm]{Corollary}
\newtheorem{lem}[thm]{Lemma}

\newtheorem{hypo}[thm]{Hypothesis}

\newtheorem{prop}[thm]{Proposition}

\theoremstyle{definition}
\newtheorem{defn}[thm]{Definition}
\newtheorem{exmp}[thm]{Example}
\newtheorem{rem}[thm]{Remark}

\newcommand\Q{{\mathbb Q}}
\newcommand\Z{{\mathbb Z}}
\newcommand\N{{\mathbb N}}
\newcommand{\Gal}{\mathop{\mathrm{Gal} }\nolimits}
\newcommand{\GL}{\mathop{\mathrm{GL} }\nolimits}
\newcommand{\GSp}{\mathop{\mathrm{GSp} }\nolimits}
\newcommand\Oh{\mathcal{O}}
\newcommand\F{\mathbb{F}}
\newcommand\fm{\mathfrak{m}}
\newcommand\wild{\mathrm{w}}
\newcommand{\Aut}{\mathop{\mathrm{Aut} }\nolimits}

\title{Formal groups, supersingular abelian varieties and tame
ramification}
\date{}

\author{Sara Arias-de-Reyna}

\begin{document}

\maketitle

\footnotetext{\hskip -0.6cm \emph{2010 Mathematics Subject
Classification:} 14L05,
11G10, 11S15\\
\emph{Key words and phrases:} tame ramification, formal group,
supersingular abelian variety\\
Research supported by a FPU predoctoral grant AP-20040601 of the MEC
and partially supported by MEC grant MTM2006-04895.}

\begin{abstract}
Let us consider an abelian variety defined over $\mathbb{Q_{\ell}}$
with good supersingular reduction. In this paper we give explicit
conditions that ensure that the action of the wild inertia group on
the $\ell$-torsion points of the variety is trivial. Furthermore we
give a family of curves of genus $2$ such that their Jacobian
surfaces have good supersingular reduction and  satisfy these
conditions.  We address this question by means of a detailed study
of the formal group law attached to abelian varieties.
\end{abstract}

\section{Introduction}

Let $\ell$ be a prime number and $A/\Q_{\ell}$ be an abelian variety
with good supersingular reduction. In this paper we  study  the
action of the wild inertia group $I_{\wild}\subset
\Gal(\overline{\Q}_{\ell}/\Q_{\ell})$ on the $\ell$-torsion points
of $A$. More precisely, we will address the problem of finding
explicit conditions that ensure that the Galois extension
$\Q_{\ell}(A[\ell])/\Q_{\ell}$ obtained by adjoining to the field of
$\ell$-adic numbers the coordinates of the $\ell$-torsion points of
$A$ is tamely ramified.

Let $E/\Q_{\ell}$ be an elliptic curve. If it has good supersingular
reduction, then the field extension $\Q_{\ell}(E[\ell])/\Q_{\ell}$
is tamely ramified (cf. \cite{Proprietesgaloisiennes}, $\S$ 1). The
proof relies on a detailed study of the formal group law attached to
$E$. This formal group law has dimension 1 and height 2. The set of
elements of $\overline{\Q}_{\ell}$ with positive $\ell$-adic
valuation can be endowed with a group structure by means of this
formal group law. Call $V$ the $\F_{\ell}$-vector space of
$\ell$-torsion points of this group (which is isomorphic to the
group of $\ell$-torsion points of $E$ as
$\Gal(\overline{\Q}_{\ell}/\Q_{\ell})$-module). One essential
ingredient in the proof is  the fact that the $\ell$-adic valuation
of the points of $V$ can be explicitly computed (see Proposition 9,
$\S$ 1.9 of \cite{Proprietesgaloisiennes}). This fact allows one to
define an embedding of  $V$  into a certain 1-dimensional
$\overline{\F}_{\ell}$-vector space (called $V_{\alpha}$ in
\cite{Proprietesgaloisiennes}) where the wild inertia group acts
trivially, and in turn this compels the wild inertia group to act
trivially upon $V$. When the dimension $n$ of the formal group law
is greater than 1 the situation becomes more complicated. It is no
longer possible to compute the $\ell$-adic valuation of the $n$
coordinates of the elements of $V$, which now denotes the group of
$\ell$-torsion points of the corresponding formal group. In this
paper we give a condition, Hypothesis \ref{H}, under which we can
prove that the wild inertia group acts trivially on $V$. The key
point is that this hypothesis allows us to define several different
embeddings of $V$ into $V_{\alpha}$.

In the rest of the paper we apply this result to the case of
dimension $2$, and produce non-trivial examples of abelian surfaces
defined over $\Q_{\ell}$ such that the ramification of
$\Q_{\ell}(A[\ell])/\Q_{\ell}$ is tame. We introduce the notion of
\emph{symmetric} $2$-dimensional formal group law, and prove that
such a formal group law satisfies Hypothesis \ref{H} under a certain
condition. Furthermore, using this result we explicitly construct,
for each $\ell\geq 5$, genus $2$ curves over $\Q_{\ell}$ such that
the formal group law attached to their Jacobians satisfy Hypothesis
\ref{H} (cf. Theorem \ref{PrimeraFamilia}). Finally we formulate a
condition that allows us to deform the curves and enlarge the family
of genus $2$ curves such that the Galois extension defined by the
$\ell$-torsion points of their Jacobians is tamely ramified, which
enables us to obtain Theorem \ref{MainResult}.

Given a prime $\ell$, in \cite{Ariasdereyna-Vila2009} the authors
contruct certain semistable elliptic curves defined over $\Q$ with
good supersingular reduction at $\ell$. When $\ell\geq 11$, these
curves provide tame Galois realizations of the group
$\GL_2(\F_{\ell})$. In this way, the authors give an affirmative
answer to the tame inverse Galois problem posed by B. Birch in
\cite{Birch}, Section 2, for the family of linear groups
$\GL_2(\F_{\ell})$.

In \cite{Paper2}, we will use the results in this paper in order to
realize the groups in the family $\GSp_4(\F_{\ell})$ as the Galois
group of a tamely ramified extension for each prime $\ell\geq 5$.

The contents of this paper are part of my Ph.D. thesis. I want to
thank my advisor, Prof. N\'uria Vila, for her constant support and
helpful conversations.

\section{Notation}

We will denote by $K$  a local field of characteristic zero and
residual characteristic $\ell$,  $v$ the corresponding discrete
valuation, normalized so that $v(K^*)=\Z$, $\Oh$ the ring of
integers of the valuation and $k$ the residue field. Further, we
will assume that $v(\ell)=1$ (that is to say,  $K$ will be an
unramified extension of $\Q_{\ell}$). We fix an algebraic closure
$\overline{K}$ of $K$, and denote by $v$ the extension of $v$ to
this algebraic closure. Finally, $\overline{k}$ denotes the
algebraic closure of $k$ obtained through the reduction of
$\Oh_{\overline{K}}$, the ring of integers of $\overline{K}$ with
respect to $v$, modulo its maximal ideal. Later in the paper, we
will take $K=\Q_{\ell}$.

We will denote by $I\subset \Gal(\overline{K}/K)$ the inertia group,
and by $I_{\wild}$ the wild inertia group.

To ease notation, we will denote the tuples of elements in boldface.
For instance, we will write $\mathbf{X}=(X_1, \dots, X_n)$,
$\mathbf{Y}=(Y_1, \dots, Y_n), \mathbf{Z}=(Z_1, \dots, Z_n)$ to
denote $n$-tuples of variables, and $\mathbf{x}=(x_1, \dots, x_n)$,
$\mathbf{y}=(y_1, \dots, y_n)$ will denote tuples of elements of
$\overline{K}$.

\section{Inertia action and the formal group law}\label{Inertia action
and the formal group}

We will start by recalling that an \emph{$n$-dimensional formal
group law} defined over $\Oh$ is a $n$-tuple of power series
\begin{equation*}(F_1(\mathbf{X}, \mathbf{Y}), \dots,
F_n(\mathbf{X}, \mathbf{Y}))\in \Oh[[X_1, \dots, X_n, Y_1, \dots,
Y_n]]^{\times n}\end{equation*} satisfying:

\begin{itemize}

\item $F_i(\mathbf{X}, \mathbf{Y})\equiv X_i + Y_i
\pmod{\textrm{terms of degree two}}$,\\ for all $i=1, \dots, n$.

\item $F_i(F_1(\mathbf{X}, \mathbf{Y}), \dots, F_n(\mathbf{X},
\mathbf{Y}), \mathbf{Z})=F_i(\mathbf{X}, F_1(\mathbf{Y},
\mathbf{Z}), \dots, F_n(\mathbf{Y}, \mathbf{Z}))$ \\ for all $i=1,
\dots, n$.

\end{itemize}

Besides, if $F_i(\mathbf{X}, \mathbf{Y})=F_i(\mathbf{Y},\mathbf{X})$
for all $i=1, \dots, n$, then the formal group law is said to be
\emph{commutative}.

To a formal power series one can attach a group. Let us denote by
$\overline{\fm}$  the set of elements of $\overline{K}$ with
positive valuation, and denote by $\overline{\fm}^{\times n}$ the
cartesian product of $\overline{\fm}$ with itself $n$ times. For
this set one can define an addition law $\oplus_F$ by
\begin{align*}
\oplus_{\mathbf{F}}:\overline{\fm}^{\times n}\times
\overline{\fm}^{\times n} &\rightarrow
\overline{\fm}^{\times n}\\
(\mathbf{x}, \mathbf{y})& \mapsto (F_1(\mathbf{x}, \mathbf{y}),
\dots, F_n(\mathbf{x}, \mathbf{y}))
\end{align*} (which is well defined since $F_i(\mathbf{x}, \mathbf{y})$
converges to an element of $\overline{\fm}$, for all $i=1, \dots
n$). The set $\overline{\fm}^{\times n}$, endowed with this sum,
turns out to be a group, which will be denoted by
$\mathbf{F}(\overline{\fm})$. Let us call $V$ the $\F_{\ell}$-vector
space of $\ell$-torsion points of $\mathbf{F}(\overline{\fm})$.

In \cite{Proprietesgaloisiennes}, $\S$ 8, an auxiliary object is
introduced.

\begin{defn}\label{Valpha} Let $\alpha\in \Q$ be a positive rational number. Consider the sets
\begin{equation*}\overline{\fm}_{\alpha}=\{x\in \overline{\fm}: v(x)\geq
\alpha\}\quad\textrm{ and }\quad \overline{\fm}_{\alpha}^+=\{x\in
\overline{\fm}: v(x)>\alpha\}.\end{equation*}

We define $V_{\alpha}$ as the quotient group
\begin{equation*}V_{\alpha}:=
\overline{\fm}_{\alpha}/\overline{\fm}^+_{\alpha}.\end{equation*}

\end{defn}

$V_{\alpha}$ has a natural structure of $\overline{k}$-vector space,
and its dimension as such is $1$. Moreover, the absolute Galois
group of $K$ acts on $V_{\alpha}$: for each $\sigma\in
\Gal(\overline{K}/ K)$, and for each $x +
\overline{\fm}^+_{\alpha}\in
\overline{\fm}_{\alpha}/\overline{\fm}^+_{\alpha}$, we have
$\sigma(x + \overline{\fm}^+_{\alpha}):=\sigma(x) +
\overline{\fm}^+_{\alpha}$. In general, this action does not respect
the $\overline{k}$-vector space structure. But if we take an element
$\sigma$ in the inertia group $I$, it induces a morphism of
$\overline{k}$-vector space on $V_{\alpha}$, and in turn this
implies that the wild inertia group $I_{\wild}$ acts trivially on
$V_{\alpha}$ (cf. $\S$ 1.8 in \cite{Proprietesgaloisiennes}). The
main point in the proof, in dimension $1$, that the wild inertia
group acts trivially on $V$ is to define an embedding of $V$ into
$V_{\alpha}$, taking advantage of the fact that the valuation of the
points of $V$ is equal to $\alpha=\frac{1}{\ell^2-1}$.

But, in the case when $n>1$, each point has $n$ coordinates, and we
have to admit the possibility that the valuations of the coordinates
of the $\ell$-torsion points of $\mathbf{F}(\overline{\fm})$ have
different values. Our idea is to formulate a weaker assumption about
the valuations of the coordinates, but which is strong enough to
imply the desired result about the action of the wild inertia group
$I_{\wild}$ on $\mathbf{F}(\overline{\fm})$.

\begin{hypo}\label{H}
There exists a positive $\alpha\in \Q$ such that, for all non-zero
$(x_1, \dots, x_n)\in V$, it holds that
\begin{equation*}\min_{1\leq i\leq n}\{v(x_i)\}=\alpha.\end{equation*}
\end{hypo}

Under this hypothesis, we are able to prove the desired result:

\begin{thm}\label{resultado general}
Let $\mathbf{F}$ be a formal group law such that the
$\F_{\ell}$-vector space $V$ of the $\ell$-torsion points of
$\mathbf{F}(\overline{\fm})$ satisfies Hypothesis \ref{H}. Then the
image of the wild inertia group $I_{\wild}$ by the Galois
representation attached to $V$ is trivial.
\end{thm}

\begin{proof}
Let $P=(x_1, \dots, x_n)\in V$. We are going to show that each
$\sigma \in I_{\wild}$ acts trivially on $P$, that is,
$\sigma(P)=P$.

According to Hypothesis \ref{H}, we have that, for each non-zero
point $Q=(y_1, \dots, y_n)\in V$,
\begin{equation*}\min_{1\leq i\leq n}\{v(y_i)\}=\alpha.\end{equation*}
Therefore, for each $n$-tuple $(\lambda_1, \dots, \lambda_n)\in
\Z^n$, we know that either $\lambda_1 y_1 + \cdots \lambda_n y_n=0$
or else it belongs to $\overline{\fm}_{\alpha}$. This allows us to
consider the following map:
\begin{equation*}\begin{aligned}
\varphi_{(\lambda_1, \dots, \lambda_n)}: V & \rightarrow
V_{\alpha}=\overline{\fm}_{\alpha}/
\overline{\fm}_{\alpha^+}\\
(y_1, \dots, y_n)& \mapsto \lambda_1y_1 + \cdots \lambda_n y_n +
\overline{\fm}_{\alpha}^+.\end{aligned}\end{equation*}

It is clear that $\varphi_{(\lambda_1, \dots, \lambda_n)}$ is a
group morphism, when we consider on $V$ the sum given by the formal
group law, and on $V_{\alpha}$ the sum induced by that of
$\overline{K}$. As a matter of fact, it is a morphism of
$\F_{\ell}$-vector spaces (for the structure of $\F_{\ell}$-vector
space is determined by the sum). Besides, it is compatible with the
Galois action.

Now let us take an element $\sigma\in I_{\wild}$. Then
\begin{equation*}\varphi_{(\lambda_1, \dots,
\lambda_n)}(\sigma(P))=\sigma(\varphi_{(\lambda_1, \dots,
\lambda_n)}(P))=\varphi_{(\lambda_1, \dots,
\lambda_n)}(P),\end{equation*} where the last equation holds because
$I_{\wild}$ acts trivially upon $V_{\alpha}$. In other words, for
each $n$-tuple $(\lambda_1, \dots, \lambda_n)\in \Z^n$,
$\sigma(P)-P$ belongs to the kernel of $\varphi_{(\lambda_1, \dots,
\lambda_n)}$. But no point of $V$ can belong to all these kernels
save the zero vector. This, again, is a consequence of Hypothesis
\ref{H}. Any non-zero point $Q=(y_1, \dots, y_n)\in V$ satisfies
that there exists $j\in \{1, \dots, n\}$ such that $v(y_j)=\alpha$.
If we take $\lambda_i=0$ for all $i\not=j$, $\lambda_j=1$, then
$\varphi_{(\lambda_1, \dots, \lambda_n)}(P)=x_j +
\overline{\fm}_{\alpha}^+\not= 0 + \overline{\fm}_{\alpha}^+$.

To sum up, for each $P\in V$ and each $\sigma\in I_{\wild}$,
$\sigma(P)-P=(0, \dots, 0)$, and so $\sigma$ acts trivially on $P$.

\end{proof}

\section{Symmetric formal group laws of dim 2}\label{Symmetric formal group laws}

Let $\mathbf{F}$ be a formal group law over $\Q_{\ell}$ of dimension
$2$. Our aim is to analyze the valuation of the $\ell$-torsion
points of $\mathbf{F}(\overline{\mathfrak{m}})$, and try to obtain
explicit conditions that ensure that Hypothesis \ref{H} holds. The
property of being $\ell$-torsion provides us with two equations in
two variables. Let us briefly recall these equations. We begin by
recalling the definition of homomorphism between formal group laws
of dimension $n$.

\begin{defn} Let $\mathbf{F}=(F_1(\mathbf{X}, \mathbf{Y}), \dots, F_n(\mathbf{X},
\mathbf{Y}))$ and \newline $\mathbf{G}=(G_1(\mathbf{X}, \mathbf{Y}),
\dots G_n(\mathbf{X}, \mathbf{Y}))$ be two formal group laws over
$\Oh$ of dimension $n$. A  \emph{homomorphism} $f$ is a $n$-tuple of
formal power series in $\Oh[[Z_1, \dots, Z_n]]$ without constant
term, say $(f_1(\mathbf{Z}), \dots,$ $f_n(\mathbf{Z}))$, such that
\begin{multline*}f(F_1(\mathbf{X}, \mathbf{Y}),\dots,  F_n(\mathbf{X},
\mathbf{Y}))=\\ =(G_1(f_1(\mathbf{X}), \dots, f_n(\mathbf{X}),
f_1(\mathbf{Y}), \dots,  f_n(\mathbf{Y})),\\ \dots,
G_n(f_1(\mathbf{X}), \dots, f_n(\mathbf{X}), f_1(\mathbf{Y}), \dots,
f_n(\mathbf{Y}))).\end{multline*}
\end{defn}

\begin{exmp}\label{multiplication by m map} For each $m\in \N$,
one can define the multiplication by $m$ map in the following way:
\begin{equation*}\begin{cases}[0](\mathbf{Z})=(0, 0, \dots, 0) \\
             [1](\mathbf{Z})=\mathbf{Z}\\
             [m+1](\mathbf{Z})=\mathbf{F}([1](\mathbf{Z}),
             [m](\mathbf{Z}))\text{ for } m\geq 1.
             \end{cases}\end{equation*}

It is easy to prove by induction that the shape of the $n$ power
series $[m]_i(\mathbf{Z})$ that constitute the multiplication by $m$
map is the following:
\begin{equation*}[m]_i(\mathbf{Z})=m\cdot Z_i +
\text{ terms of degree }\geq 2, \end{equation*} for all $i=1, \dots,
n$.
\end{exmp}

When $n=2$, the multiplication by $\ell$ map is defined by two
equations in two variables, and this complicates our attempt to
compute the valuations of the two coordinates of the points of $V$.
In order to avoid this inconvenience, we are going to restrict our
attention to a special kind of formal group laws. Namely, we will
consider formal group laws such that the two equations have a
certain relationship that allows us to reduce the problem to
studying a single equation.

\begin{defn}\label{def grupo formal symm} Let
$\mathbf{F}=(F_1(X_1, X_2, Y_1, Y_2), F_2(X_1, X_2, Y_1, Y_2))$ be a
formal group law of dimension $2$ over $\Q_{\ell}$. We will say that
$\mathbf{F}$ is a \emph{symmetric formal group law} if the following
relationship holds:
\begin{equation*}F_2(X_2, X_1, Y_2, Y_1)=F_1(X_1, X_2, Y_1, Y_2).\end{equation*}

\end{defn}

The symmetry is reflected in the power series $[\ell]_1(Z_1, Z_2)$
and $[\ell]_2(Z_1, Z_2)$. By induction on $m$, one can prove the
following lemma.

\begin{lem}\label{simetria en n} Let $\mathbf{F}(\mathbf{X}, \mathbf{Y})$
be a symmetric formal group law of dimension $2$. For all $m\geq 1$,
it holds that
\begin{equation*}[m]_2(Z_2, Z_1)=[m]_1(Z_1, Z_2).\end{equation*}
\end{lem}

Next we will establish two technical lemmas which will be useful.

\begin{lem}\label{lema tecnico 1} Let $\ell>2$ be a prime number,
$r\in \N$, and let $f(Z_1, Z_2)\in \Z_{\ell}[[Z_1, Z_2]]$ be a
formal power series such that $f(Z_2, Z_1)=-f(Z_1, Z_2)$, which can
be written as:
\begin{multline*}f(Z_1, Z_2)=  \ell\cdot (Z_1 - Z_2) +
\ell\cdot(\text{terms of total degree $\geq 2$ and $< \ell^r$})
\\ \hskip 2.3cm  + a\cdot (Z_1^{\ell^r}-Z_2^{\ell^r})+ \text{ terms of total degree
$\geq\ell^r + 1$},\end{multline*} where $\ell\nmid a$. Then if
$(x_0, y_0)\in \overline{\fm}\times \overline{\fm}$ with
$x_0\not=y_0$ satisfies $f(x_0, y_0)=0$ and furthermore $v(x_0),
v(y_0)\geq v(x_0-y_0)$, then the $\ell$-adic valuation $v(x_0 -
y_0)$ is $1/(\ell^r-1)$.
\end{lem}

\begin{proof} Let us call $\beta=v(x_0-y_0)$. We will compute the
valuations of the different terms that appear in the equality
$f(x_0, y_0)=0$.

\begin{itemize}

\item $v(\ell\cdot(x_0-y_0))= 1 + \beta$.

\item Let us consider a term of total degree between $2$ and
$\ell^r-1$, say $\ell\cdot c x_0^{n} y_0^m$. Compute its valuation:
$v(\ell\cdot c x_0^n y_0^m)= 1 + v(c) + nv(x_0) + mv(y_0)\geq 1 + (n
+ m)\beta > 1 + \beta$, since $n+m\geq 2$.

\item Let us consider the term $a (x_0^{\ell^r} - y_0^{\ell^r})$.
Let us split it into the sum of two terms, in the following way:
\begin{equation*} a\cdot(x_0^{\ell^r}- y_0^{\ell^r})= a\cdot((x_0- y_0)^{\ell^r} -
B)=a\cdot(x_0- y_0)^{\ell^r} - a\cdot B,\end{equation*} where
$B=(x_0-y_0)^{\ell^r}- (x_0^{\ell^r}- y_0^{\ell^r})$.

On the one hand, $v(a\cdot(x_0- y_0)^{\ell^r})=v(a) +
\ell^r\beta=\ell^r\beta$, since $\ell$ does not divide $a$.

On the other hand, note that \begin{multline*}(x_0-y_0)^{\ell^r}=
x_0^{\ell^r} - \binom{\ell^r}{1}x_0^{\ell^r-1}y_0 +
\binom{\ell^r}{2}x_0^{\ell^r-2}y_0^2 + \cdots \\ -
\binom{\ell^r}{2}x_0^2 y_0^{\ell^r-2}+ \binom{\ell^r}{1}x_0
y_0^{\ell^r-1} - y_0^{\ell^r}.\end{multline*}

Therefore,  each of the terms
$\binom{\ell^r}{i}(-1)^ix_0^{\ell^r-i}y_0^i$ has a valuation
strictly greater than $1+\beta$. (For $v(x_0^{\ell^r-i}y_0^i)\geq
\beta(\ell^r-i+i)=\ell^r \beta$, and hence
$v(\binom{\ell^r}{i}(-1)^ix_0^{\ell^r-i}y_0^i)\geq 1 + \beta\ell^r>
1+\beta$).

\item Since $v(x_0), v(y_0)\geq \beta$, it is clear that the
valuation of the terms of degree greater than $\ell^r$ is greater
than $\ell^r\beta$.

\end{itemize}

But obviously there must be (at least) two terms with minimal
valuation, since they must cancel out. Therefore $v(\ell\cdot
(x_0-y_0))=v(a\cdot (x_0-y_0)^{\ell^r})$, that is to say, $1 +
\beta=\ell^r\beta$, hence $\beta=1/(\ell^r-1)$, as was to be proven.

\end{proof}

\begin{lem}\label{lema tecnico 2} Let $\ell>2$ be a prime number,
$r\in \N$, and let $f(Z_1, Z_2)\in \Z_{\ell}[[Z_1, Z_2]]$ be a
formal power series such that $f(Z_2, Z_1)=f(Z_1, Z_2)$, which can
be written as:
\begin{multline*}f(Z_1, Z_2)=  \ell\cdot (Z_1 + Z_2) +
\ell\cdot(\text{terms of total degree $\geq 2$ and $< \ell^r$})\\
+ a\cdot (Z_1^{\ell^r}+Z_2^{\ell^r}) + \text{ terms of total degree
$\geq\ell^r + 1$},\end{multline*} where $\ell\nmid a$. Then if
$(x_0, y_0)\in \overline{\fm}\times \overline{\fm}$ satisfies
$f(x_0, y_0)=0$ and furthermore $v(x_0), v(y_0)\geq v(x_0 + y_0)$,
then $v(x_0 + y_0)$ is $1/(\ell^r-1)$.
\end{lem}

\begin{proof}
Analogous to that of Lemma \ref{lema tecnico 1}.
\end{proof}

We want to apply the previous lemmas to the formal power series
defined by $[\ell]_1(Z_1, Z_2)-[\ell]_2(Z_1, Z_2)$ and
$[\ell]_1(Z_1, Z_2)+[\ell]_2(Z_1, Z_2)$. In order to do this, we
need to know the value of the parameter $r$ that appears in these
formal power series. This parameter is related to the height of the
formal group law. Let us recall this notion (see \cite{Hazewinkel},
Chapter IV, (18.3.8)). Firstly, we need to define this concept for
formal group laws defined over $k$, and then we will transfer this
definition to formal group laws over $\Oh$ through the reduction
map.

\begin{defn}\label{defaltura} Let $\overline{\mathbf{F}}$ be
a formal group law of dimension n over $k$, and let
$\overline{[\ell]}=(\overline{[\ell]}_1(\mathbf{Z}), \dots,
\overline{[\ell]}_n(\mathbf{Z}))$ be the multiplication by $\ell$
map. Then $\overline{\mathbf{F}}$ is of \emph{finite height} if the
ring $k[[Z_1, \dots, Z_n]]$ is finitely generated as a module over
the subring $k[[\overline{[\ell]}_1(\mathbf{Z}), \dots,
\overline{[\ell]}_n(\mathbf{Z})]]$.

When $\overline{\mathbf{F}}$ is of finite height, it holds that
$k[[Z_1, \dots, Z_n]]$ is a free module over
$k[[\overline{[\ell]}_1(\mathbf{Z}), \dots,
\overline{[\ell]}_n(\mathbf{Z})]]$ of rank equal to a power of
$\ell$, say $\ell^h$. This $h$ shall be called the \emph{height of}
$\overline{\mathbf{F}}$.
\end{defn}

\begin{defn} Let $\mathbf{F}$ be a formal group law of dimension n over
$\Oh$. We define the \emph{height} of $\mathbf{F}$ as the height of
the reduction $\overline{\mathbf{F}}$ of $\mathbf{F}$ modulo the
maximal ideal of $\Oh$.
\end{defn}

\begin{rem}  A few words concerning the way to compute the height of a formal group law are in order. Let $\overline{f}_1(\mathbf{Z}), \dots,
\overline{f}_n(\mathbf{Z})$ be $n$ formal power series in $k[[Z_1,
\dots,  Z_n]]$ without constant term. Note that the following
statements are equivalent:
\begin{itemize}
\item $k[[Z_1, \dots, Z_n]]$ is generated by $h$ elements as a
module over the subring $k[[\overline{f}_1, \dots,
\overline{f}_n]]$.
\item $k[[Z_1, \dots,  Z_n]]/\langle\overline{f}_1, \dots,
\overline{f}_n\rangle$ is a $k$-vector space of finite dimension
less than or equal to $h$.
\end{itemize}

Therefore, to compute the height of $\overline{\mathbf{F}}$, one
seeks the least $h$ that satisfies the last property, that is, the
dimension of the $k$-vector space $$k[[Z_1, \dots,
Z_n]]/\langle\overline{f}_1, \dots,  \overline{f}_n\rangle.$$ But
this can be easily done by means of standard bases. For the
definition and some properties of standard bases in power series
rings we refer the reader to \cite{Becker1}. If $I$ is an ideal of
$k[[X_1, \dots, X_n]]$, then the dimension of $k[[X_1, \dots,
X_n]]/I$ as a $k$-vector space is determined in this way: Take a
standard basis $S$ of $I$, and consider the set of terms $M=\{t\in
T: \text{ for all }g \in S, \mathrm{LT}(g)\nmid t\}$. Then the
cardinal of $M$ is the required dimension (of course, it need not be
finite).

Now, if we have a formal group law $\overline{\mathbf{F}}$ over $k$
of dimension n, its height is the dimension of $k[[Z_1, \dots,
Z_n]]/\langle \overline{[\ell]}_1(\mathbf{Z}), \dots,
\overline{[\ell]}_n(\mathbf{Z})\rangle$, so we can compute it in an
explicit way.
\end{rem}

In the case when the formal group law is of dimension 1, another
definition of height is used (see for instance \cite{Silverman},
Chapter IV, $\S$ 7). Namely, if $\overline{F}(X, Y)$ is a formal
group law defined over $k$, the height of $\overline{F}$ is defined
as the largest $r$ such that the multiplication by $\ell$ map,
$\overline{[\ell]}(Z)$, can be expressed as
$\overline{[\ell]}(Z)=\overline{g}(Z^{\ell^r})$, for some formal
power series $\overline{g}(Z)\in k[[Z]]$. One can prove, following a
simple reasoning, that the first term of $g$ with non-zero
coefficient is precisely a constant times $Z^{\ell^r}$. Now what
happens if we try to imitate this reasoning in dimension n? As is
stated in \cite{Hazewinkel}, the reasonings in (18.3.1) can be
carried out in arbitrary dimension, yielding the following result:

\begin{prop}\label{exponente r} Let $\overline{\mathbf{F}}$, $\overline{\mathbf{G}}$ be
 formal group laws over $k$ of dimension n, and
$\overline{\mathbf{f}}:\overline{\mathbf{F}}\rightarrow\overline{\mathbf{G}}$
a non-zero homomorphism. Let us write
\begin{equation*}\overline{\mathbf{f}}(\mathbf{Z})=
(\overline{f}_1(\mathbf{Z}), \dots,
\overline{f}_n(\mathbf{Z})).\end{equation*} If $u$ is the smallest
exponent such that, in some $\overline{f}_i(\mathbf{Z})$, some
variable $Z_j$ occurs in a non-zero monomial raised to the $u$-th
power, then $u=\ell^r$ for some $r\geq 0$. Furthermore, there exist
$\overline{g}_1(\mathbf{Z}), \dots, \overline{g}_n(\mathbf{Z})\in
k[[Z_1, \dots,  Z_n]]$ such that
\begin{equation*}
\overline{f}_i(\mathbf{Z})=\overline{g}_i(\mathbf{Z}^{\ell^r}),
\text{ for all }i=1, \dots, n,\end{equation*} where
$\mathbf{Z}^{\ell^r}=(Z_1^{\ell^ r}, \dots, Z_n^{\ell^r})$.
\end{prop}

\begin{rem}\label{relacion alturas general}
We can apply this proposition to the homomorphism
$\overline{[\ell]}$ of multiplication by $\ell$ in a formal group
law $\overline{\mathbf{F}}$, and conclude that there exists an
$r\geq 0$ (in fact $r$ will be greater than or equal to 1) such that
the formal power series $\overline{[\ell]}_i(\mathbf{Z})$, $i=1,
\dots, n$, can be expressed as formal power series in the variables
$Z_1^{\ell^r}$, $\dots, Z_n^{\ell^r}$. But this $r$ might not be
determined by the height of $\overline{\mathbf{F}}$. For instance,
it might be the case that the height of $\overline{\mathbf{F}}$ is
infinite, while the exponent $r$ must always be a finite number. The
following proposition deals with this matter.
\end{rem}

\begin{prop}\label{prop relacion alturas}
Let $\overline{\mathbf{F}}$ be a $2$-dimensional formal group law
defined over $\F_{\ell}$, and assume that there exist two power
series in $\F_{\ell}[[Z_1, Z_2]]$, say $\overline{f}_1,
\overline{f}_2$, such that the formal power series that give
multiplication by $\ell$ map $\overline{[\ell]}$ can be written as
\begin{equation*}\begin{cases}\overline{[\ell]}_1(Z_1,
Z_2)=\overline{f}_1(Z_1^{\ell^r}, Z_2^{\ell^r}),\\
\overline{[\ell]}_2(Z_1, Z_2)=\overline{f}_2(Z_1^{\ell^r},
Z_2^{\ell^r}).\end{cases}\end{equation*} Then the height of
$\overline{\mathbf{F}}$ is greater than or equal to $2r$.
\end{prop}

\begin{proof}
Let us write
\begin{equation*}\begin{cases} \overline{f}_1(Z_1, Z_2)=a_{11}Z_1 +
a_{12}Z_2 + \text{terms of degree $\geq 2$}\\ \overline{f}_2(Z_1,
Z_2)=a_{21}Z_1 + a_{22}Z_2 + \text{terms of degree $\geq
2$.}\end{cases}\end{equation*} We may assume that one element (at
least) of the set $\{a_{11}, a_{12}, a_{21}, a_{22}\}$ does not
vanish, say $a_{11}\not=0$ (the other cases are analogous).

Consider the graduated lexicographical ordering on $\F_{\ell}[[Z_1,
Z_2]]$ with $Z_1<Z_2$, that is to say, the relation $\leq$
determined by the following rules:
\begin{equation*}Z_1^{a}Z_2^b < Z_1^cZ_2^d \leftrightarrow
\begin{cases} a+b<c+d \text{ or }\\
                 a+b=c+d \text{ and }a>c.\end{cases}\end{equation*}

Let $I$ be the ideal generated by $\overline{f}_1(Z_1, Z_2)$ and
$\overline{f}_2(Z_1, Z_2)$. In order to compute the height of
$\overline{\mathbf{F}}$, we need to find a standard basis for $I$.
Now the smallest monomial with respect to this ordering is $Z_1$.
And this monomial appears in $\overline{f}_1(Z_1, Z_2)$. We can
therefore use it to eliminate all monomials under a given degree of
$\overline{f}_2(Z_1, Z_2)$, save those which are pure in $Z_2$. In
fact, if $\overline{f}_2(Z_1, Z_2)$ is not a multiple of
$\overline{f}_1(Z_1, Z_2)$, we will reach a point where the power
series $\overline{g}_2(Z_1, Z_2)$ obtained from $\overline{f}_2$ by
eliminating the terms divisible by $Z_1$ up to a certain degree has
as leading term a monomial which is pure in $Z_2$, say
$\overline{g}_2(Z_1, Z_2)= b_{0, t}Z_2^t + \text{ terms of degree
$\geq t+1$}$. Then it is easily seen that $\{\overline{f}_1,
\overline{g}_2\}$ is a standard basis for $I$, and the rank of
$\F_{\ell}[[Z_1, Z_2]]/I$ as a $\F_{\ell}$-module is $t$.

Recall that the height of $\overline{\mathbf{F}}$ is  the rank of
$\F_{\ell}[[Z_1, Z_2]]/\langle \overline{[\ell]}_1,
\overline{[\ell]}_2\rangle$. Clearly this rank is
$\ell^{r}\cdot(\ell^r t)=\ell^{2r}t$.  But we know that $t$ must be
a power of $\ell$ (see Definition \ref{defaltura}), say $t$ is of
the form $\ell^s$ for some $s\in \N$. Hence the height of
$\overline{\mathbf{F}}$ is $2r + s$, which is greater than (or equal
to) $2r$.

\end{proof}

\begin{rem}\label{relacion alturas} Note that the height of a formal group law of dimension $2$ must be comprised between $2$ and $4$. Actually, the case that interests us is when the height is $4$.
In this case, only two possibilities might occur:

\begin{itemize}

\item The exponent $r$ in Proposition \ref{exponente r} is $2$. By
Proposition \ref{prop relacion alturas}, there exists an $s\in \N$
such that $4=2r + s=4+s$. Hence $s=0$.

\item The exponent $r$ in Proposition \ref{exponente r} is $1$.
Then by Proposition \ref{prop relacion alturas}, there exists an
$s\in \N$ such that $4=2r + s=2+s$. Hence $s=2$.

\end{itemize}

Assume $s=0$. If we write the multiplication by $\ell$ map as
\begin{equation*}\begin{cases} \overline{[\ell]}_1(Z_1, Z_2)=\overline{a} Z^{\ell^2}_1 +
\overline{b}Z^{\ell^2}_2 + \text{terms of degree $\geq \ell^2$}\\
\overline{[\ell]}_2(Z_1, Z_2)=\overline{c}Z^{\ell^2}_1 +
\overline{d}Z^{\ell^2}_2 + \text{terms of degree $\geq
\ell^2$}\end{cases}\end{equation*} then the determinant of the
matrix $\begin{pmatrix}\overline{a} & \overline{b} \\ \overline{c} &
\overline{d}\end{pmatrix}$ is non-zero.
\end{rem}

We will finally state and prove the main theorem of this section:

\begin{thm}\label{thm valoraciones 1} Let $\ell>2$ be a prime number and
let $\mathbf{F}=(F_1, F_2)$ be a $2$-dimensional symmetric formal
group law over $\Z_{\ell}$. Assume it has height 4 and the exponent
in Proposition \ref{exponente r} is $r=2$. Let us denote by $V$ the
$\F_{\ell}$-vector space of $\ell$-torsion points of
$\mathbf{F}(\overline{\fm})$, $\alpha=1/(\ell^2-1)$.

Then for all $(x_0, y_0)\in V$,
\begin{equation*}\min\{v(x_0), v(y_0)\}=\alpha.\end{equation*}

\end{thm}

\begin{proof} First of all, let us recall that,
since the formal group law $\mathbf{F}$ is symmetric and of height 4
with $r=2$, Remark \ref{relacion alturas} allows us to write the two
formal power series that comprise the multiplication by $\ell$ map
in the following way:
\begin{equation*}\begin{cases}[\ell]_1(Z_1, Z_2)=
\ell Z_1 + \ell\cdot(\text{terms of total degree $\geq 2$ and $<
\ell^2$}) \\
\hskip 2.3cm + a\cdot Z_1^{\ell^2} + b\cdot Z_2^{\ell^2} + \text{
terms of
degree $\geq\ell^2 + 1$}\\
[\ell]_2(Z_1, Z_2)= \ell Z_2 + \ell\cdot(\text{terms of total degree
$\geq 2$ and $< \ell^2$})\\ \hskip 2.3cm + b\cdot Z_1^{\ell^2} +
a\cdot Z_2^{\ell^2} +   \text{ terms of degree $\geq\ell^2 + 1$}
\end{cases}\end{equation*} with $\ell\nmid a^2 - b^2$.

Take a point  $P=(x_0, y_0)\in V$.  We split the proof in two cases.

\item \textbf{Case 1}: $v(x_0)\not=v(y_0)$. Assume that
$v(x_0)<v(y_0)$ (otherwise we proceed analogously). Then
$v(x_0-y_0)=v(x_0)$. We will apply Lemma \ref{lema tecnico 1} with
$r=2$. The point $(x_0, y_0)$ satisfies both equations
$[\ell]_1(x_0, y_0)=0$ and $[\ell]_2(x_0, y_0)=0$. Therefore it also
satisfies that $f(x_0, y_0)=[\ell]_1(x_0, y_0) - [\ell]_2(x_0,
y_0)=0$. Furthermore, taking into account the previous
considerations, we can write
\begin{multline*}f(Z_1, Z_2)= \ell (Z_1 - Z_2) +\\
+\ell\cdot(\text{terms of total degree $\geq 2$ and $< \ell^2$})+
(a-b)\cdot (Z_1^{\ell^2} - Z_2^{\ell^2})+\\ +  \text{ terms of
degree greater than or equal to $\ell^2 + 1$},
\end{multline*} and $\ell\nmid a-b$. Nothing prevents us
now from applying Lemma \ref{lema tecnico 1} and concluding that
$v(x_0-y_0)=\alpha$. But then $\alpha=v(x_0)<v(y_0)$, hence
$\min\{v(x_0), v(y_0)\}=\alpha$.

\item \textbf{Case 2:} $v(x_0)=v(y_0)$. Then either $v(x_0 -
y_0)=v(x_0)$ or $v(x_0 + y_0)=v(x_0)$. (For both must be greater
than or equal to $v(x_0)$. And taking into account that
$\ell\not=2$, we obtain $v(x_0)=v(2x_0)=v((x_0 + y_0) + (x_0 -
y_0))$, so both $v(x_0 + y_0)$ and $v(x_0 + y_0)$ cannot be greater
than $v(x_0)$). If  $v(x_0 - y_0)=v(x_0)$, we can apply Lemma
\ref{lema tecnico 1} as in the previous case and conclude that
$v(x_0)=v(y_0)=\alpha$. If $v(x_0 + y_0)=v(x_0)$, we make use of
Lemma \ref{lema tecnico 2} with $f=[\ell]_1 + [\ell]_2$ and $r=2$,
thus concluding that $v(x_0)=v(y_0)=\alpha$. This completes the
proof.

\end{proof}

Combining this theorem with Theorem \ref{resultado general}, we
obtain the following result:

\begin{thm}\label{thm grupo formal dim 2} Let $\ell>2$ be a prime number, and
let $\mathbf{F}=(F_1, F_2)$ be a $2$-dimensional symmetric formal
group law over $\Z_{\ell}$. Assume it has height 4 and the exponent
in Proposition \ref{exponente r} is $r=2$. Then the wild inertia
group $I_{\wild}$ acts trivially on the $\F_{\ell}$-vector space of
$\ell$-torsion points of $\mathbf{F}(\overline{\fm})$.
\end{thm}

\section{Symmetric genus $2$ curves}\label{Symmetric genus $2$ curves}

In this section we are going to present a certain kind of genus 2
curves such that their Jacobians are abelian surfaces with good
supersingular reduction, and moreover the corresponding formal group
law satisfies the hypotheses of Theorem \ref{thm grupo formal dim
2}. Let us fix an odd prime number $\ell$.

\begin{defn}\label{def curva simetrica} We shall call a genus $2$ curve \emph{symmetric} if it
can be expressed through an equation $y^2=f(x)$, where $f(x)=f_0 x^6
+ f_1 x^5 + f_2 x^4 + f_3 x^3 + f_2 x^2 + f_1 x + f_0$ is a
polynomial of degree 6 and non-zero discriminant.
\end{defn}

In my PhD thesis \cite{Tesis} the following result is proven.

\begin{thm}\label{corolario simetria} Let $f(x)=f_0 x^6 + f_1 x^5 + f_2 x^4 + f_3 x^3 + f_2
x^2 + f_1 x + f_0\in \mathbb{Q}_{\ell}[x]$ be a polynomial of degree
6 and non-zero discriminant, and let $\mathbf{F}=(F_1, F_2)$ be the
formal group law attached to the Jacobian variety of the curve
defined by $y^2=f(x)$. Then
\begin{equation*}F_2(s_2, s_1, t_2, t_1)=F_1(s_1, s_2, t_1, t_2).\end{equation*}
\end{thm}

In this way we can control the symmetry of the formal group law.
With respect to the height, it is well known that the formal group
law attached to an abelian surface with good supersingular reduction
has height 4 (cf. \cite{p-divisibleGroups}). We will say that a
genus $2$ curve defined over $\F_{\ell}$ is supersingular if its
Jacobian is a supersingular abelian surface.

Our aim is to construct, for a given  prime number $\ell>3$, a
symmetric genus $2$ curve over $\Q_{\ell}$ with supersingular
reduction. In fact, what we shall construct is a supersingular genus
$2$ curve, defined over $\F_{\ell}$ by an equation
$y^2=\overline{f}(x)$, where $\overline{f}(x)= \overline{f}_0x^6 +
\overline{f}_1x^5 + \overline{f}_2x^4 + \overline{f}_3x^3 +
\overline{f}_2x^2 + \overline{f}_1 x + \overline{f}_0\in
\F_{\ell}[x]$ is a polynomial of degree $6$ with non-zero
discriminant. Lifting this equation to $\Q_{\ell}$ in a suitable way
we will obtain the curve we were seeking.

Fix $\ell>3$, and assume we have a supersingular elliptic curve $E$
defined by  $y^2=x^3 + b x^2 + b x + 1$ for a certain $b\in
\F_{\ell}$. Then the bielliptic curve $C$ defined by the equation
$y^2=x^6 + b x^4 + b x^2 + 1$ is a supersingular genus $2$ curve.
For the discriminant $\Delta_f$ of $f(x)=x^6 + bx^4 + b x^2 + 1$ and
the discriminant $\Delta_g$ of $g(x)=x^3 + bx^2 + bx + 1$ are
related by the equation $\Delta_f=-64\Delta_g$ and the
characteristic of our base field is different from $2$. On the other
hand, $C$ is isogenous to $E\times E$ (cf. \cite{Prolegomena},
Chapter 14), hence the supersingularity of $C$. Therefore, our
problem boils down to finding a supersingular elliptic curve defined
by an equation of the form $y^2=x^3 + b x^2 + b x + 1$.

Recall that an elliptic curve in Legendre form
$y^2=x(x-1)(x-\lambda)$ defined over a finite field of
characteristic $\ell$ is supersingular if and only if
$H_{\ell}(\lambda)=0$, where
$H_{\ell}(x)=\sum_{k=0}^{\frac{\ell-1}{2}}
\binom{\frac{\ell-1}{2}}{k}^2 x^k$ is the Deuring polynomial (see
Theorem 4.1-(b) in Chapter IV of \cite{Silverman}). Moreover,  there
is always a quadratic factor of $H_{\ell}(x)$ of the form $x^2-x+a$
for a certain $a\in \F_{\ell}^*$, provided $\ell>3$ (see Theorem
1-(b) of \cite{Brillhart-Morton}, cf. Corollary 3.6 of
\cite{Ariasdereyna-Vila2009}). We exploit this fact in the following
proposition.

\begin{prop}\label{CurvaExplicita} Let $a\in \F_{\ell}$ be such that $x^2-x+a$ divides $H_{\ell}(x)$. Then the equation
\begin{equation*}y^2=x^3 + \frac{1-a}{a}x^2 + \frac{1-a}{a}x +
1\end{equation*} defines a supersingular elliptic curve over
$\F_{\ell}$.
\end{prop}

\begin{proof} The discriminant of $g(x)=x^3 + \frac{1-a}{a}x^2 + \frac{1-a}{a}x +
1$  is $\Delta_g=-\frac{(-1 + 4 a)^3}{a^4}$, which does not vanish
(if $\Delta_g=0$, then $a=1/4$, and the polynomial $x^2-x+a$ would
have a double root. But the Deuring polynomial $H_{\ell}(x)$ does
not have double roots). Moreover, one can easily transform this
equation into Legendre form with $\lambda=\frac{1}{2} +
\frac{\sqrt{1-4a}}{1}$.
\end{proof}

\begin{rem} Assume $\ell=3$. The only supersingular elliptic curve over
$\F_{3}$ is given by the equation $y^2=x(x-1)(x+1)$. We can study
all the changes of variables which turn this equation into a
symmetric one, but we only obtain the curve given by $y^2=x^3+1$,
which is a singular curve. Therefore, there is no symmetric
polynomial $f(x)\in \F_3[x]$ such that the curve defined by
$y^2=f(x)$ is a supersingular elliptic curve. This is the reason why
we exclude the prime $\ell=3$ from our reasonings.
\end{rem}

In order to apply Theorem \ref{thm grupo formal dim 2} to the curves
provided by Proposition \ref{CurvaExplicita}, we need to check that
the exponent in Proposition \ref{exponente r} is $r=2$. Let us work
with the reductions of the Jacobians. First of all, note that this
property is preserved by  isogenies of degree prime to the
characteristic $\ell$.

\begin{lem}\label{separable isogeny} Let $A$ and $B$ be abelian varieties defined over $k$,
and $\Phi:B\rightarrow A$ an isogeny of degree prime to $\ell$.
Assume moreover that the formal group law attached to $B$ has $r=2$.
Then the formal group law attached to $A$ has $r=2$ too.
\end{lem}
\begin{proof} Let $m$ be the degree of $\Phi$. We know
that there exists an isogeny $\Psi:A\rightarrow B$ (the dual isogeny
of $\Phi$) such that $\Psi\circ \Phi=\overline{[m]}_B$.

Consider the following commutative diagram:
\begin{equation*}\xymatrix{B\ar[r]^{\overline{[\ell]}_B} \ar[d]^{\Phi} & B\ar[d]^{\Phi}\\
                           A\ar[r]^{\overline{[\ell]}_A} & A}\end{equation*}
Since $\Phi\circ\overline{[\ell]}_B=\overline{[\ell]}_A\circ\Phi$,
$\Phi\circ\overline{[\ell]}_B\circ \Psi=
\overline{[\ell]}_A\circ\Phi\circ\Psi$; and thus
$\Phi\circ\overline{[\ell]}_B\circ \Psi=\overline{[\ell]}_A\circ
\overline{[m]}_A$.

Consider now the homomorphism these arrows induce on the formal
group laws on $A$ and $B$ (we will not change their names). Since
$\overline{[\ell]}_B$ modulo $\ell$ can be expressed by means of
formal power series in $Z_1^{\ell^2}, Z_2^{\ell^2}$, the same is
true of the composition $\Phi\circ
\overline{[\ell]}_B\circ\Psi=\overline{[\ell]}_A\circ\overline{[m]}_A$.
But since the multiplication by $m$ map in the formal group law of
$A$ is defined by
\begin{equation*}\begin{cases}\overline{[m]}_1(Z_1, Z_2)=m Z_1 + \cdots\\
\overline{[m]}_2(Z_1, Z_2)=m Z_2 + \cdots
\end{cases}\end{equation*}
neither of the formal power series that define $\overline{[\ell]}_A$
can possess a term of degree smaller than $\ell^2$ (for $m$ is
invertible in $\F_{\ell}$). Taking into account Proposition
\ref{exponente r}, we conclude that the multiplication by $\ell$ map
in $A$ must also be expressible as a formal power series in
$Z_1^{\ell^2}, Z_2^{\ell^2}$.
\end{proof}

We will now see that the natural isogeny from $E\times E$ to the
Jacobian of $C$ (cf. \cite{Prolegomena}, Chapter 14) satisfies the
conditions of the lemma above. We will make use of the following
result (cf. Proposition 3 of \cite{HoweLeprevostPoonen}).

\begin{prop}\label{Prop3}
Let $E$ and $F$ be two elliptic curves over $\F_{\ell}$, let $A$ be
the polarized abelian surface $E\times F$, and let $G\subset
A[2](\overline{\F}_{\ell})$ be the graph of a group isomorphism
$\psi:E[2](\overline{\F}_{\ell})\rightarrow
F[2](\overline{\F}_{\ell})$. Then $G$ is a maximal isotropic
subgroup of $A[2](\overline{\F}_{\ell})$, and furthermore the
quotient polarized abelian variety $A/G$ is isomorphic to the
Jacobian of a curve $C$ over $\overline{\F}_{\ell}$, unless $\psi$
is the restriction to $E[2](\overline{\F}_{\ell})$ of an isomorphism
$E\rightarrow F$ over $\overline{\F}_{\ell}$. Moreover, the curve
$C$ and the isomorphisms are defined over $\F_{\ell}$ if $\psi$ is
an isomorphism of $\Gal(\overline{\F}_{\ell}/\F_{\ell})$-modules.
\end{prop}

Let us consider the elliptic curve $E$ defined by the Weierstrass
equation $y^2=x^3 + bx^2 + bx+1$. The  $2$-torsion points of $E$ are
the following:
\begin{align*}O&\ \\
                P_1&:=(-1, 0)\\
                P_2&:=(\frac{1}{2}(1 - b + \sqrt{-3 - 2 b + b^2}), 0)\\
                P_3&:=(\frac{1}{2}(1 - b - \sqrt{-3 - 2 b + b^2}), 0).\\
\end{align*}

Let us consider the group morphism
$\psi:E[2](\overline{\F}_{\ell})\rightarrow
E[2](\overline{\F}_{\ell})$ defined as
\begin{equation*}O\mapsto O, \hskip 0.2cm P_1\mapsto P_1, \hskip
0.2cm P_2\mapsto P_3, \hskip 0.2cm P_3\mapsto P_2.\end{equation*}

Note that it is compatible with the action of
$\Gal(\overline{\F}_{\ell}/\F_{\ell})$. In order to apply
Proposition \ref{Prop3}, we need to check that $\psi$ is not induced
from an automorphism of $E$.

But the group of automorphisms  of $E$ is well known (cf.
\cite{Silverman}, Chapter III, $\S$ 10). Namely, if $E$ is an
elliptic curve with $j$-invariant different from $0$ or $1728$ (that
is to say, with $b$ different from $0$ or $-3/2$), then the group of
automorphisms of $E$ has order $2$, and the non-trivial automorphism
corresponds to $(x, y)\mapsto(x, -y)$. Therefore, it cannot restrict
to the morphism $\psi$. In the other cases, the order of $\Aut(E)$
is $4$ or $6$: it is easy to compute these automorphisms explicitly
and check that they cannot restrict to $\psi$.

Therefore, for each $b\in \F_{\ell}$ such that the equation $y^2=x^3
+ bx^2 + bx + 1$ defines an elliptic curve $E$ (i.e., $b\not=3,
-1$), Proposition \ref{Prop3} tells us that there exists a genus $2$
curve $C$ and an isogeny
\begin{equation*}\Phi:E\times E\rightarrow J(C)\end{equation*}
which is separable (because of the definition of the quotient of
abelian varieties, cf. $\S$ 7 Chapter 2, Theorem on p. 66 of
\cite{Mumford}) of degree 4. Moreover, the isogeny can be defined
over $\F_{\ell}$. Therefore, if $E$ is a supersingular elliptic
curve we can apply Lemma \ref{separable isogeny} and conclude that
the Jacobian of $C$ satisfies that the exponent in Proposition
\ref{exponente r} is $2$. But can $C$ be explicitly determined?
Fortunately, Proposition 4 of \cite{HoweLeprevostPoonen} gives a
very explicit recipe for computing $C$. As a conclusion, we can
state the following result.

\begin{prop}\label{prop r igual a 2} Let $b\in \F_{\ell}$ be such that the Weierstrass
equation $y^2=x^3 + bx^2 + bx + 1$ defines a supersingular elliptic
curve over $\F_{\ell}$. Then the formal group law attached to the
Jacobian of the genus $2$ curve $C$ defined by a lifting of the
hyperelliptic equation
\begin{equation*}y^2=x^6 + bx^4 + bx^2 + 1\end{equation*}
has exponent $r=2$.
\end{prop}

This provides us with all the ingredients to give a family of genus
$2$ curves such that the action of the wild inertia group on the
$\ell$-torsion points of their Jacobians is trivial.

\begin{thm}\label{PrimeraFamilia} Let $\ell>3$ be a prime number. Let $\overline{a}\in\F_{\ell}$ be such that $x^2 - x + \overline{a}$ divides the Deuring polynomial $H_{\ell}(x)$, and lift it to $a\in \Z_{\ell}$. Let $f_0, f_1, f_2, f_3\in \Z_{\ell}$ such that $f_0- 1, f_1, f_2- (1-a)/a, f_3\in (\ell)$. Then the equation
$y^2= f_0x^6 + f_1 x^5 + f_2 x^ 4 + f_3 x^3 + f_2 x^2 + f_1 x +
f_0\in \Z_{\ell}[x]$ defines a genus $2$ curve $C$ such that the
Galois extension $\Q_{\ell}(J(C))/\Q_{\ell}$ is tamely ramified.

\end{thm}

\section{Approximation to symmetry}\label{Approximation}

The results in the previous section provide, for each $\ell>3$, a
symmetric genus $2$ curve with good supersingular reduction such
that its formal group law satisfies the hypotheses of Theorem
\ref{thm grupo formal dim 2}, and in consequence also the hypotheses
of Theorem \ref{resultado general}. But one might argue that these
curves are not a good example to illustrate Theorem \ref{resultado
general}, in the sense that they are actually isogenous over
$\Q_{\ell}$ to a product of elliptic curves with good supersingular
reduction, and surely one can prove in a more direct fashion that
the wild inertia group at $\ell$ acts trivially. Our aim now is to
enlarge this class of curves, and provide other more complicated
examples in which Theorem \ref{resultado general} applies. The key
idea is that we are going to take  curves which are ``approximately
symmetric'', that is to say, symmetric up to a certain order with
respect to the $\ell$-adic valuation. More specifically, we wish to
determine how close the coefficients of a hyperelliptic equation of
$C'$ must be to those of a hyperelliptic symmetric equation for the
condition in Hypothesis \ref{H} to be preserved. The main result of
this section is the following.

\begin{thm}\label{RelajarSimetria}
Let  $C$ be a genus $2$ curve given by a hyperelliptic equation
\begin{equation*}y^2=f_6 x^6 + f_5 x^5 + f_4 x^4 + f_3 x^3 +
f_2 x^2 + f_1 x + f_0,\end{equation*} where $f_0, \dots, f_6\in
\Z_{\ell}$, and consider the genus $2$ curve $C'/\Q_{\ell}$ given by
the equation
\begin{equation*}y^2=f_6' x^6 + f_5' x^5 + f_4' x^4 + f_3' x^3 +
f_2' x^2 + f_1' x + f_0'\end{equation*} with ${f_0'}, \dots, f_6'\in
\Z_{\ell}$ and satisfying $f_i-f_i'\in (\ell^4)$. Then if the formal
group law attached to the Jacobian of $C$ satisfies Hypothesis
\ref{H} with $\alpha=\frac{1}{\ell^2-1}$, so does the formal group
law attached to the Jacobian of $C'$.
\end{thm}

The rest of the section is devoted to proving this result. Fix  a
genus $2$ curve $C/\Q_{\ell}$, given by a hyperelliptic equation
\begin{equation*}y^2=
f_6 x^6 + f_5 x^5 + f_4 x^4 + f_3 x^3 + f_2 x^2 + f_1 x +
f_0,\end{equation*} where $f_0, \dots, f_6\in \Z_{\ell}$, and
consider the genus $2$ curve $C'/\Q_{\ell}$ given by the equation
\begin{equation*}y^2=
f_6' x^6 + f_5' x^5 + f_4' x^4 + f_3' x^3 + f_2' x^2 + f_1' x +
f_0'\end{equation*} with ${f_0'}, \dots, f_6'\in \Z_{\ell}$.

Denote by $\mathbf{F}=(F_1, F_2)$ (resp. $\mathbf{F}'=(F_1', F_2')$)
the formal group law attached to $C$ (resp. $C'$). It can be proven
that the coefficients of $F_i$ (resp. $F_i'$) lie in $\Z[f_0, \dots,
f_6]$ (resp. $\Z[f_0', \dots, f_6']$), $i=1, 2$.

Therefore, if we assume that, for all $i=0, \dots, 6$, $f_i-f_i'\in
(\ell^s)$, then the difference $F_i(s_1, s_2, t_1, t_2)-F_i'(s_1,
s_2, t_1, t_2)$ has coefficients in $(\ell^s)$. Hence we may drop
the curves and work in the formal group setting, since all we have
to determine is the exponent $s$ which preserves Hypothesis \ref{H}.

Denote by $\overline{\Q}_{\ell}$ an algebraic closure of
$\Q_{\ell}$, and $\overline{\fm}\subset \overline{\Q}_{\ell}$ the
set of elements with positive valuation. If the coefficients of the
power series  $[\ell]_1(Z_1, Z_2)$, $[\ell]_2(Z_1, Z_2)$ are close
(with respect to the $\ell$-adic valuation) to the coefficients of
the series $[\ell]'_1(Z_1, Z_2), [\ell]'_2(Z_1, Z_2)$, does this
imply that the solutions of the system of equations $[\ell]_1(Z_1,
Z_2)=[\ell]_2(Z_1, Z_2)=0$ are close to the solutions of the system
of equations $[\ell]'_1(Z_1, Z_2)=[\ell]'_2(Z_1, Z_2)=0$?

A precise answer to this question can be found in \cite{Bourbaki},
chapter III, $\S$ 4, n$^{\circ}$ 5. The reasoning is carried out in
the context of restricted formal power series, but it can be adapted
to this setting.

Namely, let $A$ be a commutative ring, and fix an ideal $\fm$ of
$A$. Assume that $A$ is separable and complete with respect to the
$\fm$-adic topology. As usual, we will denote the tuples of elements
in boldface.

Consider a system of $n$ power series in $n$ variables,
\begin{equation*}\mathbf{f}=(f_1, \dots, f_n),\quad
f_i\in A[[X_1, \dots, X_n]].\end{equation*}

We will denote by $J_{\mathbf{f}}$ the determinant of the Jacobian
matrix, that is to say,
\begin{equation*} J_{\mathbf{f}}=
\mathrm{det}\begin{pmatrix}\frac{\partial f_1}{\partial X_1} & \cdots & \frac{\partial f_1}{X_n}\\
\cdots & \cdots & \cdots\\
\frac{\partial f_n}{\partial X_1} & \cdots & \frac{\partial
f_n}{X_n}\end{pmatrix}.\end{equation*}  By $\fm^{\times n}$ we shall
mean the cartesian product of $\fm$ with itself $n$ times. We will
say that two $n$-tuples $\mathbf{a}$ and $\mathbf{b}$ are congruent
modulo an ideal $I$ of $A$ if they are so coordinatewise, that is to
say, $a_i-b_i\in I$ for $i=1, \dots, n$. We will apply the following
result (cf. Corollary 1 in \cite{Bourbaki}, chapter III, $\S$ 4,
n$^{\circ}$ 5).

\begin{cor}\label{Above} Let $\mathbf{f}=(f_1, \dots, f_n)$ be a tuple of
elements in $A[[X_1, \dots, X_n]]$, and let $\mathbf{a}\in
\fm^{\times n}$. Call $e=J_{\mathbf{f}}(\mathbf{a})$. If
$\mathbf{f}(\mathbf{a})\equiv 0\mod e^2\fm$, then there exists
$\mathbf{b}\in \fm^{\times n}$ such that $\mathbf{f}(\mathbf{b})=0$
and $\mathbf{b}\equiv \mathbf{a}\mod e\fm$. Furthermore, assume that
there exists another tuple $\mathbf{b}'\in \fm^{\times n}$ such that
$\mathbf{f}(\mathbf{b}')=0$ and $\mathbf{b}'\equiv \mathbf{a} \mod
e\fm$. Then, if $A$ has no zero divisors, $\mathbf{b}=\mathbf{b}'$.
\end{cor}

Let us go back now to our approximation problem. We have two formal
group laws $\mathbf{F}$, $\mathbf{F}'$, defined over $\Z_{\ell}$. We
consider the two systems of equations
\begin{equation}\label{Sistemas aproximacion}\begin{cases}[\ell]_1(Z_1, Z_2)=0\\
                              [\ell]_2(Z_1, Z_2)=0\end{cases}  \text{ and }
\begin{cases}[\ell]'_1(Z_1, Z_2)=0\\
                              [\ell]'_2(Z_1, Z_2)=0\end{cases}
\end{equation}
where we know that for $i=1, 2$, it holds that
\begin{equation*}[\ell]_i(Z_1, Z_2) - [\ell]_i'(Z_1, Z_2)\in \ell^s\cdot
\Z_{\ell}[[Z_1, Z_2]].\end{equation*}

Furthermore, since the systems of equations \eqref{Sistemas
aproximacion} describe the $\ell$-torsion points of the Jacobians of
curves of genus $2$, the set of solutions in $\overline{\fm}^{\times
2}$ is finite. We may thus consider a finite extension $K\supset
\Q_{\ell}$ that contains all the coordinates of all the solutions of
the systems in \eqref{Sistemas aproximacion}. Let us denote by
$\Oh_K$ the ring of integers of $K$ and by $\fm$ its maximal ideal.
It is clear that $\Oh_K$ is separable and complete with respect to
the $\fm$-adic topology.

Let us call $V'$ the set of pairs $(x', y')\in\overline{\fm}\times
\overline{\fm}$ such that $[\ell]'_1(x', y')=[\ell]'_2(x', y')=0$.
Our first claim is the following:

\begin{lem} For all $(x', y')\in V'$, $[\ell]_1(x', y'),
[\ell]_2(x', y')\in \ell^s\fm$.
\end{lem}

\begin{proof} Since $[\ell]'_1(x', y')=0$, we can write \begin{equation*}[\ell]_1(x',
y')=[\ell]_1(x', y')- [\ell]'_1(x', y').\end{equation*}

Furthermore, let us express $$[\ell]_1(x, y)=\sum_{ij}a_{ij}x^i y^j
\text{ and } [\ell]'_1(x, y)=\sum_{ij}a'_{ij}x^i y^j.$$ Hence
$[\ell]_1(x', y')=\sum_{ij}(a_{ij}-a'_{ij}) {x'}^i {y'}^j$. We know
that $a_{ij}- a'_{ij}\in (\ell^s)$, and ${x'}, {y'}\in \fm$, and
also that $[\ell]_1(x, y)$ is a power series without constant term;
thus it follows that $[\ell]_1({x'}, {y'})\in \ell^s \fm$. A similar
reasoning shows that $[\ell]_2({x'}, {y'})\in \ell^s \fm$.
\end{proof}

In order to apply Corollary \ref{Above} to the system of equations
$[\ell]_1(Z_1, Z_2)=$ $[\ell]_2(Z_1, Z_2)=0$, we need to compute the
determinant of the Jacobian matrix
$e=\mathrm{det}\begin{pmatrix}\ell & 0 \\ 0 &
\ell\end{pmatrix}=\ell^2$. This suggests that we should choose
$s=4$.

\begin{proof}[Proof of Theorem \ref{RelajarSimetria}]
Take $(x', y')\in \overline{\fm}^{\times 2}$ satisfying the
equations $$[\ell]'_1(x', y')=[\ell]'_2(x', y')=0.$$ We know that
$[\ell]_1(x', y'), [\ell]_2(x', y')\in \ell^4\cdot\fm$. Hence there
exists a unique $(x, y)\in \overline{\fm}^{\times 2}$ such that
$[\ell]_1(x, y)=[\ell]_2(x, y)=0$ and furthermore
\begin{equation*}\begin{cases}{x'}\equiv x \mod \ell^2 \fm\\
             {y'}\equiv y \mod \ell^2 \fm.\end{cases}\end{equation*}
In particular, the two conditions $v({x'}-x)\geq 2$, $v({y'}-y)\geq
2$ are satisfied.

But $(x, y)$ is a point of $\ell$-torsion of the Jacobian of $C$,
and therefore we know that
\begin{equation*}\min\{v(x), v(y)\}=\alpha=\frac{1}{\ell^2-1}.\end{equation*}

But if $v(x)=\alpha$ and $v({x'}-x)\geq 2>\alpha$, then it follows
that $v({x'})=\alpha$. And similarly, if $v(y)=\alpha$, then
$v({y'})=\alpha$. Also if $v(x)>\alpha$, it cannot happen that
$v(x')<\alpha$ (and the same applies to $y, y'$). We may conclude
that $\min\{v({x'}), v({y'})\}=\alpha$.
\end{proof}

Gathering together Proposition \ref{PrimeraFamilia} and Theorem
\ref{RelajarSimetria} we obtain, for each prime $\ell>3$, a large
family of abelian surfaces such that the action of the wild inertia
group upon their $\ell$-torsion points is trivial.

\begin{thm}\label{MainResult} Let $\ell>3$ be a prime number. Let $\overline{a}\in\F_{\ell}$ be such that $x^2 - x + \overline{a}$ divides the Deuring polynomial $H_{\ell}(x)$, and lift it to $a\in \Z_{\ell}$. Let $f_0, f_1, \dots, f_6\in \Z_{\ell}$ satisfy  that $f_6-f_0, f_5-f_1, f_4-f_2\in (\ell^4)$ and furthermore $f_6- 1, f_5, f_4- (1-a)/a, f_3\in (\ell)$. Then the equation
$y^2= f_6x^6 + f_5 x^5 + f_4 x^ 4 + f_3 x^3 + f_2 x^2 + f_1 x +
f_0\in \Z_{\ell}[x]$ defines a genus $2$ curve $C$ such that the
Galois extension $\Q_{\ell}(J(C))/\Q_{\ell}$ is tamely ramified.

\end{thm}

\bibliographystyle{amsalpha}

\end{document}